\documentclass[preprint,11pt]{elsarticle}
\usepackage[T1]{fontenc}
\usepackage[utf8]{inputenc}
\usepackage[english]{babel}
\usepackage{amsmath,amsthm,enumerate}
\usepackage{amssymb}
\usepackage{amscd}
\usepackage{vmargin}
\usepackage{url}
\usepackage{stmaryrd}
\usepackage{tikz,pgf,tkz-berge}

\newcommand\blank[1][.6em]{%
  \mbox{\kern.06em\vrule height.5ex}%
  \vbox{\hrule width#1}%
  \hbox{\vrule height.5ex}}
  
\begin{document}
\begin{frontmatter}

\title{Disjoint cycles of different lengths in graphs and digraphs\tnoteref{thanks}}
\tnotetext[thanks]{The first author was supported by ERC Advanced Grant GRACOL, project no. 320812.
The second author was supported by an FQRNT postdoctoral research grant and CIMI research fellowship.}

\author[dtu]{Julien Bensmail}
\author[toulouse]{Ararat Harutyunyan}
\author[lip]{Ngoc Khang Le}
 
\author[china]{\\ Binlong Li}
\author[salon]{Nicolas Lichiardopol}

\address[dtu]{Department of Applied Mathematics and Computer Science \\ Technical University of Denmark \\ DK-2800 Lyngby, Denmark \\~}
\address[toulouse]{Institut de Math\'ematiques de Toulouse \\ Université Toulouse III \\ 31062 Toulouse Cedex 09, France\\~}
\address[lip]{Laboratoire d'Informatique du Parall\'elisme \\ \'Ecole Normale Sup\'erieure de Lyon \\ 69364 Lyon Cedex 07, France\\~}
\address[china]{Department of Applied Mathematics \\ Northwestern Polytechnical University \\ Xi'an, Shaanxi 710072, P.R. China\\~}
\address[salon]{Lyc\'{e}e A. de Craponne \\ Salon, France\\~}

\begin{abstract}
Understanding how the cycles of a graph or digraph behave in general has always been an important point of graph theory.
In this paper, we study the question of finding a set of $k$ vertex-disjoint cycles (resp. directed cycles) of distinct lengths
in a given graph (resp. digraph). In the context of undirected graphs, we prove that, for every $k \geq 1$, every graph with minimum degree at least $\frac{k^2+5k-2}{2}$
has~$k$ vertex-disjoint cycles of different lengths, where the degree bound is best possible.
We also consider stronger situations, and exhibit degree bounds (some of which are best possible) when \textit{e.g.} the graph is triangle-free, or 
the $k$ cycles are requested to have different lengths congruent to some values modulo some~$r$.
In the context of directed graphs, we consider a conjecture of Lichiardopol concerning the least minimum out-degree
required for a digraph to have $k$ vertex-disjoint directed cycles of different lengths.
We verify this conjecture for tournaments, and, by using the probabilistic method,
for regular digraphs and digraphs of small order.
\end{abstract}

\end{frontmatter}

\newtheorem{theorem}{Theorem}
\newtheorem{lemma}[theorem]{Lemma}
\newtheorem{conjecture}[theorem]{Conjecture}
\newtheorem{observation}[theorem]{Observation}
\newtheorem{claim}[theorem]{Claim}
\newtheorem{corollary}[theorem]{Corollary}
\newtheorem{proposition}[theorem]{Proposition}
\newtheorem{question}[theorem]{Question}
\numberwithin{theorem}{section}

%%%%%%%%%%%%%%%%%%%%%%%%%%%%%%%%%%%%%%%%%%%%%%%%%%%%%%%%%%%%%%%%%%%%%%%%%%%%%%%%%%%%%%%%%%%%%%%%%%%%%%%%%%%%%%
%%%%%%%%%%%%%%%%%%%%%%%%%%%%%%%%%%%%%%%%%%%%%%%%%%%%%%%%%%%%%%%%%%%%%%%%%%%%%%%%%%%%%%%%%%%%%%%%%%%%%%%%%%%%%%
%%%%%%%%%%%%%%%%%%%%%%%%%%%%%%%%%%%%%%%%%%%%%%%%%%%%%%%%%%%%%%%%%%%%%%%%%%%%%%%%%%%%%%%%%%%%%%%%%%%%%%%%%%%%%%
%%%%%%%%%%%%%%%%%%%%%%%%%%%%%%%%%%%%%%%%%%%%%%%%%%%%%%%%%%%%%%%%%%%%%%%%%%%%%%%%%%%%%%%%%%%%%%%%%%%%%%%%%%%%%%
%%%%%%%%%%%%%%%%%%%%%%%%%%%%%%%%%%%%%%%%%%%%%%%%%%%%%%%%%%%%%%%%%%%%%%%%%%%%%%%%%%%%%%%%%%%%%%%%%%%%%%%%%%%%%%
%%%%%%%%%%%%%%%%%%%%%%%%%%%%%%%%%%%%%%%%%%%%%%%%%%%%%%%%%%%%%%%%%%%%%%%%%%%%%%%%%%%%%%%%%%%%%%%%%%%%%%%%%%%%%%

\section{Introduction} \label{section:introduction}

The notion of cycles has a fundamental role in many notions and problems of graph theory, in both its undirected and 
directed contexts. Hence studying the behaviour of cycles in graphs and digraphs led to many interesting and appealing works and investigations.
As the literature on this topic is obviously quite wide, pointing out some particular results and directions of research would be irrelevant.
For that reason, let us just mention, at the attention of the interested reader, the survey~\cite{AG85} on cycles in undirected graphs
by Alspach and Godsil, and the survey~\cite{BT81} on directed cycles in digraphs by Bermond and Thomassen.

\medskip

In this paper, we study degree conditions guaranteeing the existence in a graph (resp. digraph)
of a certain number of vertex-disjoint cycles (resp. directed cycles) whose lengths verify particular properties.
More precisely, not only we want cycles (resp. directed cycles) being vertex-disjoint, but we also request their lengths to be different somehow.
Namely, we first ask for the lengths to be different only, but then also request additional properties on the lengths such as having 
the same remainder to some modulo.

We start in Section~\ref{section:undirectedg} by studying how the number of vertex-disjoint cycles of different lengths in an undirected 
graph behaves in front of the minimum degree of that graph. More precisely, we consider, given some $k \geq 1$, the minimum degree 
required for a graph to have at least~$k$ vertex-disjoint cycles of different lengths (and sometimes additional length properties). 
We show that this value is precisely $\frac{k^2+5k-2}{2}$ for every $k$ (Theorem~\ref{ThMinimumDegree}). 
Several more constrained situations are then considered, \textit{e.g.} when the graph is triangle-free or the vertex-disjoint cycles are requested
to be more than just of different lengths; in these situations as well, we exhibit bounds (most of which are tight) on the least minimum degree required
to guarantee the existence of the $k$ desired vertex-disjoint cycles.
We also consider the opposite direction, and conjecture that for every $D \geq 3$, every graph $G$
verifying $k+1 \leq \delta(G) \leq \Delta(G) \leq D$ and of large enough order has $k$ vertex-disjoint cycles
of different lengths (see Conjecture~\ref{ConjMaximumDegree}).
To support this conjecture, we prove it for $k=2$ (Theorem~\ref{ThConjMaximumDegree}).
This in particular yields that every cubic graph of order more than $14$ has two vertex-disjoint cycles
of different lengths, which is tight (see Theorem~\ref{ThCubicGraph}).

We then consider, in Section~\ref{section:directed}, the same question for digraphs: What minimum out-degree is required for a 
digraph to have at least~$k$ vertex-disjoint directed cycles of different lengths? 
The existence of such a minimum out-degree was conjectured by Lichiardopol in~\cite{Lic14}, who verified it for $k=2$.
We here give further support to Lichiardopol's Conjecture by showing it to hold for tournaments (see Corollary~\ref{corollary:tournaments}),
and, using the probabilistic method, for regular digraphs (Theorem~\ref{theorem:regular}) and digraphs of small order (Theorem~\ref{theorem:order}).

\section{Disjoint cycles of different lengths in undirected graphs} \label{section:undirectedg}

In this section, we consider the existence of disjoint cycles of different lengths in graphs.
We start off, in Section~\ref{section:prelim}, by recalling a few results and introducing new results and concepts of independent interest.
We then prove our main results in Section~\ref{section:undirected}. 

\subsection{Preliminaries} \label{section:prelim}

Let $G$ be a graph and $X$ a subset of $V(G)$. We use $G[X]$ to
denote the subgraph of $G$ induced by $X$, and $G-X$ to denote the
subgraph of $G$ induced by $V(G)\backslash X$. For two disjoint
subsets $X,Y$ of $V(G)$, we denote by $(X,Y)_G$ the bipartite subgraph of
$G$ with all edges between $X$ and $Y$. For a subgraph $H$ of $G$,
we set $G-H=G-V(H)$.

\subsubsection{Known results}

By considering the last vertex of a longest path of some graph, we get following.

\begin{proposition}\label{ThCDL}
For every $k\geq 1$, every graph of minimum degree at least $k+1$ contains
$k$ cycles of different lengths.
\end{proposition}

By considering a maximum cut of some graph, we clearly get the following.

\begin{proposition}\label{ThBSMD}
Let $G$ be a graph of minimum degree at least $2k-1$, where $k\geq
1$. Then $V(G)$ can be partitioned into sets $S$ and $T$ such that
the bipartite subgraph $(S,T)_G$ has minimum degree at least $k$.
\end{proposition}

An immediate consequence is:

\begin{theorem}\label{ThECDL}
For every $k\geq 1$, every graph of minimum degree at least $2k+1$
contains $k$ even cycles of different lengths.
\end{theorem}

\begin{proof}
By Proposition \ref{ThBSMD}, we can partition $V(G)$ into sets $S$ and
$T$ such that the bipartite subgraph $G'=(S,T)_G$ has minimum degree at least $k+1$. By Proposition \ref{ThCDL}, graph $G'$
contains $k$ cycles of different lengths. Since $G'$ is bipartite, all these cycles have even length.
\end{proof}
\medskip

Proposition \ref{ThBSMD} shows the existence of a cut of any graph
such that every vertex has `many' neighbors in the different partite
set. In the different flavour, the following three theorems concern cuts of graphs
such that every vertex has `many' neighbors in the same partite set.

\begin{theorem}[Stiebitz \cite{Stiebitz}]\label{ThSt}
If $s$ and $t$ are non-negative integers, and $G$ is a graph with
minimum degree at least $s+t+1$, then the vertex set of $G$ can be
partitioned into two sets which induce subgraphs of minimum degree
at least $s$ and $t$, respectively.
\end{theorem}

\begin{theorem}[Kaneko \cite{Kaneko}]\label{ThKa}
Let $s$ and $t$ be integers with $s\geq 1$ and $t\geq 1$. Then for
every triangle-free graph $G$ of minimum degree at least $s+t$,
there exists a partition $(S,T)$ of $V(G)$ such that the induced
subgraph $G\lbrack S\rbrack$ is of minimum degree at least $s$, and the
induced subgraph $G\lbrack T\rbrack$ is of minimum degree at least $t$.
\end{theorem}

\begin{theorem}[Diwan \cite{Diwan}]\label{ThDi}
Let $s$ and $t$ be integers with $s\geq 2$ and $t\geq 2$. Then for
every  graph $G$ of girth at least $5$ and of minimum degree at
least $s+t-1$, there exists a partition $(S,T)$ of $V(G)$ such that
the induced subgraph $G\lbrack S\rbrack$ is of minimum degree at least
$s$, and the induced subgraph $G\lbrack T\rbrack$ is of minimum degree
at least $t$.
\end{theorem}

Diwan proved in \cite{Diwan'} the following result:

\begin{theorem}[Diwan \cite{Diwan'}]\label{ThDi'}
For every $r\geq 2$, and for any natural number $m$, every graph of
minimum degree at least $2r-1$ contains a cycle of length $2m$
modulo $r$.
\end{theorem}

It is easy to seen that when $r$ is odd, for any natural number $m$,
every graph of minimum degree at least $2r-1$ contains a cycle of
length $m$ modulo $r$.

\begin{corollary}\label{CoThDi'}
For every $k\geq 1$ and $r\geq 2$, every graph $G$ with
$$\delta(G)\geq\left\{\begin{array}{ll}
2(k+1)r-1,  & \mbox{if } k \mbox{ is even and } r \mbox{ is odd};\\
2kr-1,      & \mbox{otherwise},
\end{array}\right.$$
has $k$ cycles of different lengths all divisible by $r$.
\end{corollary}

\begin{proof}
If $r$ is even, then $\delta(G)\geq 2kr-1$. By Theorem \ref{ThDi'},
there exist $k$ cycles $C_0,C_1,\ldots,C_{k-1}$ such that for each
$i$, cycle $C_i$ has length $ir$ (mod $kr$). Clearly the $k$ cycles have
different lengths all divisible by $r$.

If both $k$ and $r$ are odd, then $\delta(G)\geq 2kr-1$. By Theorem
\ref{ThDi'}, there exist $k$ cycles $C_0,C_1,\ldots,C_{k-1}$ such
that for each $i$, cycle $C_i$ has length $2ir$ (mod $kr$). Since $kr$ is
odd, the $k$ cycles have different lengths all divisible by $r$.

If $k$ is even and $r$ is odd, then $\delta(G)\geq 2(k+1)r-1$. Note
that $k+1$ is odd. By the analysis above, graph $G$ contains $k+1$ cycles
of different lengths all divisible by $r$.
\end{proof}

%%%%%%%%%%%%%%%%%%%%%%%%%%%%%%%%%%%%%%%%%
%%%%%%%%%%%%%%%%%%%%%%%%%%%%%%%%%%%%%%%%%
%%%%%%%%%%%%%%%%%%%%%%%%%%%%%%%%%%%%%%%%%
%%%%%%%%%%%%%%%%%%%%%%%%%%%%%%%%%%%%%%%%%
%%%%%%%%%%%%%%%%%%%%%%%%%%%%%%%%%%%%%%%%%
%%%%%%%%%%%%%%%%%%%%%%%%%%%%%%%%%%%%%%%%%

\subsubsection{Graphs with only one type of cycles}

In the upcoming results, we use $n(G)$ to denote the order of some graph $G$. For integer
$i$, we denote by $N_i(G)$ the set of vertices of $G$ with degree
$i$, and set $n_i(G)=|N_i(G)|$. 
Using these notions, we prove the following two results concerning graphs with only one particular type of cycles.
These results will mainly be used in the proof of Theorem~\ref{ThConjMaximumDegree}.

\begin{lemma}\label{LeCycleTriangle}
Let $G$ be a graph with $\delta(G)\geq 2$. If all cycles of $G$ have length~$3$, then $n_2(G)\geq \dfrac{n(G)}{3}+2$.
\end{lemma}

\begin{proof}
We first claim that every block (maximal $2$-connected subgraph) of $G$ is either a $K_2$ or a
triangle. Let $B$ be an arbitrary block of $G$. If $B\neq K_2$, then
$B$ is 2-connected. If $B$ has at least four vertices, then, by
Dirac's Theorem, which states that a block with minimum degree~$k$ has a cycle of length at least~$2k$, 
block $B$ contains a cycle of order at least 4, which is
not in our assumption. So we conclude that $B$ has exactly three
vertices. Thus $B$ is a triangle.

Now we prove the lemma by induction on the number of blocks of $G$. If
$G$ has only one block, then $G$ is a triangle and the assertion is
trivial. If $G$ is disconnected, then we can complete the proof by
applying the induction hypothesis to each component of $G$. Now we
assume that $G$ is connected and separable.

Let $B$ be an end-block of $G$, and $x_0$ be the cut-vertex of $G$
contained in $B$. Recall that $B$ is a triangle. If
$d_{G-B}(x_0)\geq 2$, then let $G'=G-(B-x_0)$. By the induction
hypothesis, we have $n_2(G')\geq \dfrac{n(G')}{3}+2$. Note that
$n(G)=n(G')+2$ and $n_2(G)\geq n_2(G')+1$. We then have $$n_2(G)\geq
n_2(G')+1\geq\frac{n(G')}{3}+2+1=\frac{n(G)-2}{3}+3\geq\frac{n(G)}{3}+2.$$
Now we assume that $d_{G-B}(x_0)=1$. Let $P=(x_0,x_1,\cdots,x_t)$ be
a path of $G$ such that $d(x_i)=2$ for $1\leq i\leq t-1$ and
$d(x_t)\geq 3$. Let $G'=G-B-(P-x_t)$. By the induction hypothesis, we have
$n_2(G')\geq \dfrac{n(G')}{3}+2$. Note that $n(G)=n(G')+t+2$ and
$n_2(G)\geq n_2(G')+t$. Since $t\geq 1$, we now have $$n_2(G)\geq
n_2(G')+t\geq\frac{n(G')}{3}+2+t=\frac{n(G)-t-2}{3}+t+2\geq
\frac{n(G)}{3}+2,$$ which concludes the proof.
\end{proof}
\medskip

\begin{lemma}\label{LeCycleQuadrangle}
Let $G$ be a graph with $\delta(G)\geq 2$. If all cycles of $G$ have length~$4$, then $n_2(G)\geq \dfrac{n(G)}{5}+2$.
\end{lemma}

\begin{proof}
We first claim that every block of $G$ is either a $K_2$ or a
complete bipartite graph $K_{2,s}$ with $s\geq 2$. Let $B$ be an
arbitrary block of $G$. If $B\neq K_2$, then $B$ is 2-connected.
Thus $B$ contains a cycle, which is a $C_4$ by our assumption. If
$B$ is not bipartite, then $B$ has an odd cycle, which is not a
$C_4$, a contradiction. This implies that $B$ is bipartite. Let
$X,Y$ be the two partite sets of $B$ and let
$C=(x_1,y_1,x_2,y_2,x_1)$ be a $C_4$ of $G$, where $x_1,x_2\in X$
and $y_1,y_2\in Y$. Suppose that $X$ has a third vertex $x_3$. Since
$B$ is 2-connected, there is a path $P$ with two end-vertices in $C$
such that $x_3\in V(P)$ and all internal vertices of $P$ are not in
$C$. If $P$ has length more than 2, then $B$ contains a cycle longer
than $C$, which is not a $C_4$, a contradiction. Thus $P$ has length
2 and $x_3y_1,x_3y_2\in E(B)$. If $Y$ has a third vertex $y_3$, then
by a similar analysis as above, we can see that $x_1y_3,x_2y_3\in
E(G)$. But in this case $C'=(x_1,y_3,x_2,y_2,x_3,y_1,x_1)$ is a
cycle being not a $C_4$, a contradiction. This implies that either
$|X|=2$ or $|Y|=2$. We suppose without loss of generality that
$|X|=2$. Since $B$ is 2-connected, vertices $x_1,x_2$ are adjacent to every
vertex in $Y$. Thus $B$ is a complete bipartite graph $K_{2,s}$ with
$s\geq 2$.

Now we prove the lemma by induction on the block number of $G$. If
$G$ has only one block, then $G=K_{2,s}$ with $s\geq 2$ and the
assertion is trivial. If $G$ is disconnected, then we can complete
the proof by applying the induction hypothesis to each component of
$G$. Now we assume that $G$ is connected and separable.

Let $B$ be an end-block of $G$, and $x_0$ be the cut-vertex of $G$
contained in $B$. Note that $B=K_{2,s}$ with $s\geq 2$. If
$d_{G-B}(x)\geq 2$, then let $G'=G-(B-x)$. By the induction hypothesis,
$n_2(G')\geq \dfrac{n(G')}{5}+2$. Note that $n(G)=n(G')+s+1$ and
$$n_2(G)\geq\left\{\begin{array}{ll}
n_2(G')+2,  & \mbox{if } s=2;\\
n_2(G')+s-2,    & \mbox{if } s\geq 3.
\end{array}\right.$$ One can compute that $n_2(G)\geq \dfrac{n(G)}{5}+2$. Now we
assume that $d_{G-B}(x)=1$. Let $P=(x_0,x_1,\cdots, x_t)$ be a path
of $G-(B-x_0)$ such that $d(x_i)=2$ for $1\leq i\leq t-1$ and
$d(x_t)\geq 3$. Let $G'=G-B-(P-x_t)$. By the induction hypothesis,
we have $n_2(G')\geq \dfrac{n(G')}{5}+2$. Note that $n(G)=n(G')+s+t+1$ and
$$n_2(G)\geq\left\{\begin{array}{ll}
n_2(G')+t+1,  & \mbox{if } s=2;\\
n_2(G')+s+t-3,    & \mbox{if } s\geq 3,
\end{array}\right.$$
where $t\geq 1$. This yields $n_2(G)\geq
\dfrac{n(G)}{5}+2$.
\end{proof}

%%%%%%%%%%%%%%%%%%%%%%%%%%%%%%%%%%%%%%%%%
%%%%%%%%%%%%%%%%%%%%%%%%%%%%%%%%%%%%%%%%%
%%%%%%%%%%%%%%%%%%%%%%%%%%%%%%%%%%%%%%%%%
%%%%%%%%%%%%%%%%%%%%%%%%%%%%%%%%%%%%%%%%%
%%%%%%%%%%%%%%%%%%%%%%%%%%%%%%%%%%%%%%%%%
%%%%%%%%%%%%%%%%%%%%%%%%%%%%%%%%%%%%%%%%%

\subsubsection{Path-vertex schemas}

In one of our proofs below (namely, Theorem~\ref{ThLargeOrder}), we will make use of the following notion, 
which we believe is of independent interest.
Consider an integer $k\geq 2$ and a graph $G$ of minimum degree
$\delta(G) \geq k+1$. By a \textit{$k$-path-vertex schema} of $G$ we mean a pair
$\mathcal{S}=(P,x)$ consisting of a path $P$ and a vertex $x$ not in
$P$ having exactly $k+1$ neighbors in $P$. Using a longest path argument, 
it is easy to see that every graph contains a path-vertex schema. Observe
also that the subgraph induced by $V(P)\cup\{x\}$ contains $k$
cycles of different length (all containing $x$). For convenience, we
sometimes consider $\mathcal{S}$ as a subgraph of $G$. So
$V(\mathcal{S})=V(P)\cup\{x\}$, and
$G-\mathcal{S}=G-(V(P)\cup\{x\})$. The \textit{cardinality} of $\mathcal{S}$
is $\vert V(\mathcal{S}) \rvert$. It is obvious that this cardinality is at
least $k+2$. An \textit{optimal} $k$-path-vertex schema is a schema of
minimum cardinality.
Observe that in this case the extremities of $P$ are neighbors of $x$.

We will make use of the following result concerning optimal path-vertex schemas.

\begin{lemma}\label{LeOptimalSchema}
Let $G$ be a graph with $\delta(G) \geq k+1$, and $\mathcal{S}=(P,x)$ be an optimal $k$-path-vertex schema of $G$. 
Then every vertex $y$ of $G-\mathcal{S}$ has at most $k+2$ neighbors in
$V(\mathcal{S})$. Moreover, if there exists a vertex $y$ in
$G-\mathcal{S}$ having exactly $k+2$ neighbors in $V(\mathcal{S})$,
then $x$ is a neighbor of $y$, all the vertices of $P$ are neighbors
of $x$, and $|V(\mathcal{S})|=k+2$.
\end{lemma}

\begin{proof}
Suppose $P=(u_1,u_2,\ldots,u_p)$. If $y$ has at least $k+2$
neighbors in $V(P)$, then there exists $i$ (with $2\leq i\leq p$) such
that $y$ has exactly $k+1$ neighbors in $P'=(u_i,\ldots,u_{p})$.
Thus $P'$ and $y$ form a $k$-path-vertex schema of cardinality less
than $\mathcal{S}$, a contradiction. This implies that $y$ has at
most $k+2$ neighbors in $V(\mathcal{S})$.

Suppose that $y$ has exactly $k+2$ neighbors in $V(\mathcal{S})$. By
the analysis above, we can see that $y$ has exactly $k+1$ neighbors
on $P$ and $xy\in E(G)$. Thus $(P,y)$ is an optimal $k$-path-vertex
schema, and then $u_1$ and $u_p$ are neighbors of $y$. Suppose for
the sake of a contradiction that there exists a vertex of $P$ which
is not a neighbor of $x$. Clearly $u_1$ and $u_p$ are neighbors of
$x$. Let $i$ be the maximum index such that $xu_i\notin E(G)$. Thus
$2\leq i\leq p-1$. If $i\leq p-2$, then let
$P'=(u_{i+2},\ldots,u_p,y,u_1,\ldots,u_{i-1})$; otherwise, if $i=p-1$, let
$P'=(y,u_1,\ldots,u_{p-2})$. Then $P'$ and $x$ form a
$k$-path-vertex schema of cardinality less than $\mathcal{S}$, a
contradiction. So, all vertices of $P$ are neighbors of $x$, and
clearly we have $|V(\mathcal{S})|=k+2$.
\end{proof}

%%%%%%%%%%%%%%%%%%%%%%%%%%%%%%%%%%%%%%%%%%%%%
%%%%%%%%%%%%%%%%%%%%%%%%%%%%%%%%%%%%%%%%%%%%%
%%%%%%%%%%%%%%%%%%%%%%%%%%%%%%%%%%%%%%%%%%%%%
%%%%%%%%%%%%%%%%%%%%%%%%%%%%%%%%%%%%%%%%%%%%%
%%%%%%%%%%%%%%%%%%%%%%%%%%%%%%%%%%%%%%%%%%%%%
%%%%%%%%%%%%%%%%%%%%%%%%%%%%%%%%%%%%%%%%%%%%%
%%%%%%%%%%%%%%%%%%%%%%%%%%%%%%%%%%%%%%%%%%%%%
%%%%%%%%%%%%%%%%%%%%%%%%%%%%%%%%%%%%%%%%%%%%%
%%%%%%%%%%%%%%%%%%%%%%%%%%%%%%%%%%%%%%%%%%%%%
%%%%%%%%%%%%%%%%%%%%%%%%%%%%%%%%%%%%%%%%%%%%%
%%%%%%%%%%%%%%%%%%%%%%%%%%%%%%%%%%%%%%%%%%%%%

\subsection{Main results for disjoint cycles of different lengths in undirected graphs} \label{section:undirected}

We now consider the following question: 
What is the minimum degree $f(k)$ required
so that a graph with minimum degree $f(k)$ has at least $k$ vertex-disjoint
cycles of different lengths?
In particular, we exhibit the best possible function.  

\begin{theorem}\label{ThMinimumDegree}
{\bf a)} For every $k \geq 1$, there exists a minimum integer
$f(k)$ such that every graph of minimum degree at least $f(k)$
contains $k$ vertex-disjoint cycles of different lengths. \newline
{\bf b)} We have $f(k)=\dfrac{k^2+5k-2}{2}$.
\end{theorem}

\begin{proof}
{\bf a)} We proceed by induction on $k$. Clearly $f(1)$ exists (and
we have $f(1)=2$). Suppose that the assertion is true for $k-1$ (where $k\geq 2$), 
and let us study it for $k$. Let $G$ be an arbitrary graph
of minimum degree at least $f(k-1)+k+2$. By Theorem~\ref{ThSt},
there exists a partition $(V_1,V_2)$ of $V(G)$ such that $G\lbrack
V_1\rbrack$ is of minimum degree at least $f(k-1)$ and $G\lbrack
V_2\rbrack$ is of minimum degree at least $k+1$. Then $G\lbrack
V_1\rbrack$ contains $k-1$ disjoint cycles $C_1,\ldots, C_{k-1}$ of
different lengths. By Proposition~\ref{ThCDL}, subgraph $G\lbrack V_2\rbrack$
contains $k$ cycles of different lengths. Then one of these cycles
has length distinct from those of the cycles $C_1,\ldots, C_{k-1}$.
We get then a collection of $k$ vertex-disjoint cycles of different
lengths. So the assertion is true for $k$, and the result is
proved. Furthermore, we have $f(k)\leq f(k-1)+k+2$.

\medskip

{\bf b)} We have $f(i)\leq f(i-1)+i+2$ for $2\leq i\leq k$. By
addition and simplification we get $f(k)\leq f(1)+4+\cdots + k+2$,
hence $f(k)\leq 2+\dfrac{(k+2)(k+3)}{2}-6$, which yields
\setcounter{equation}{0}
\begin{align}
f(k)\leq \dfrac{k^2+5k-2}{2}.
\end{align}
On the other hand since the complete graph on $f(k)+1$ vertices is
of minimum degree $f(k)$, it contains $k$ vertex-disjoint cycles of
different lengths. It follows that $f(k)+1\geq 3+\cdots+(k+2)$, hence
$f(k)+1\geq\dfrac{(k+2)(k+3)}{2}-3$, which yields
\begin{align}
f(k)\geq \dfrac{k^2+5k-2}{2}.
\end{align}
From Inequalities (1) and (2), we get $f(k)=\dfrac{k^2+5k-2}{2}$.
\end{proof}
\medskip

For triangle-free graphs, the function $f(k)$ from Theorem~\ref{ThMinimumDegree}
can be refined to the following, where the exhibited function ($g(k)$ in the following statement) is again best possible.

\begin{theorem}\label{ThTriangleFree}
{\bf a)} For every $k \geq 1$, there exists a minimum integer
$g(k)$ such that every triangle-free graph of minimum degree at
least $g(k)$ contains $k$ vertex-disjoint cycles of different
lengths.
\newline {\bf b)} We have $g(k)=\dfrac{k(k+3)}{2}$.
\end{theorem}

\begin{proof}
{\bf a)} We proceed by induction on $k$. Clearly $g(1)$ exists (and
$g(1)=2$). Suppose that the assertion is true for $k-1$, where
$k\geq 2$. Let $G$ be an arbitrary graph of minimum degree at least
$g(k-1)+k+1$. By Theorem \ref{ThKa}, there exists a partition
$(V_1,V_2)$ of $V(G)$ such that $G\lbrack V_1\rbrack$ is of minimum
degree at least $g(k-1)$ and $G\lbrack V_2\rbrack$ is of minimum
degree at least $k+1$. Then $G\lbrack V_1\rbrack$ contains $k-1$
vertex-disjoint cycles $C_1,\ldots, C_{k-1}$ of different lengths.
By Proposition \ref{ThCDL}, subgraph $G\lbrack V_2\rbrack$ contains $k$ cycles of
different lengths. Then one of these cycles has length distinct from
those of the cycles $C_1,\ldots, C_{k-1}$. We therefore get a collection of
$k$ vertex-disjoint cycles of different lengths.

\medskip

{\bf b)} By the analysis above, we have $g(i)\leq g(i-1)+i+1$ for
$2 \leq i \leq k$. By addition and simplification we get $g(k)\leq
g(1)+3+4+\cdots+(k+1)$. That is,
\begin{align}
g(k)\leq \frac{k(k+3)}{2}.
\end{align}
On the other hand, since the complete bipartite graph
$K_{g(k),g(k)}$ is triangle-free and of minimum degree $g(k)$, it
contains $k$ vertex-disjoint cycles of different lengths. It follows that
$2g(k)\geq 4+6+\cdots+2(k+1)$, which yields
\begin{align}
g(k)\geq\frac{k(k+3)}{2}.
\end{align}
Thus, from Inequalities~(3) and~(4) we get $g(k)=\dfrac{k(k+3)}{2}$.
\end{proof}
\medskip

We now consider the impact on the function $f(k)$ of Theorem~\ref{ThMinimumDegree}
if we additionally require the vertex-disjoint cycles to be of even lengths only.
The function $h(k)$ we exhibit below is best possible.

\begin{theorem}\label{ThEvenCycle}
{\bf a)} For every $k \geq 1$, there exists a minimum integer
$h(k)$ such that every graph of minimum degree at least $h(k)$
contains $k$ vertex-disjoint even cycles of different lengths.
\newline {\bf b)} We have $h(k)=k^2+3k-1$.
\end{theorem}

\begin{proof}
{\bf a)} We proceed by induction on $k$. Clearly $h(1)$ exists (and,
by Theorem \ref{ThECDL}, we have $h(1)=3$). Suppose that the
assertion is true for $k-1$, for some $k\geq 2$, and let us study it for $k$.
Let $G$ be an arbitrary graph of minimum degree at least
$h(k-1)+2k+2$. By Theorem \ref{ThSt}, there exists a partition
$(V_1,V_2)$ of $V(G)$ such that $G\lbrack V_1\rbrack$ is of minimum
degree at least $h(k-1)$ and $G\lbrack V_2\rbrack$ is of minimum
degree at least $2k+1$. Then $G\lbrack V_1\rbrack$ contains $k-1$
vertex-disjoint even cycles $C_1,\ldots, C_{k-1}$ being of different
lengths. According to Theorem~\ref{ThECDL}, subgraph $G\lbrack V_2\rbrack$ contains $k$
even cycles of different lengths. Then one of these cycles has
length distinct from those of the cycles $C_1,\ldots, C_{k-1}$. We
get then a collection of $k$ vertex-disjoint even cycles of
different lengths. So, the assertion is true for $k$, and the result
is proved. We also deduce $h(k)\leq h(k-1)+2k+2$.

\medskip

{\bf b)} We have $h(i)\leq h(i-1)+2i+2$ for $2 \leq i \leq k$. By addition
and simplification, we get $h(k)\leq h(1)+2(3+\cdots + k+1)$, hence
$h(k)\leq 3+2\left(\dfrac{(k+1)(k+2)}{2}-3\right)$, which yields
\begin{align}
h(k)\leq k^2+3k-1.
\end{align}
On the other hand since the complete graph on $h(k)+1$ vertices is
of minimum degree $h(k)$, it contains $k$ vertex-disjoint even
cycles of different lengths. It follows that $h(k)+1\geq 2\cdot
2+\cdots+2(k+1)$, hence $h(k)+1\geq
2\left(\dfrac{(k+1)(k+2)}{2}-1\right)$, which yields
\begin{align}
h(k)\geq k^2+3k-1.
\end{align}
We thus deduce $h(k)=k^2+3k-1$ from Inequalities~(5) and~(6).
\end{proof}
\medskip

We extend Theorems \ref{ThMinimumDegree} and \ref{ThEvenCycle} as
follows.

\begin{theorem}\label{ThLengthDivisible}
{\bf a)} Let $r\geq 3$ be an integer. For every $k \geq 1$, there
exists a minimum integer $f_r(k)$ such that every graph of minimum
degree at least $f_r(k)$ contains $k$ vertex-disjoint cycles of
different lengths all divisible by $r$. \newline {\bf b)} We have
$f_r(k)\leq\left\{\begin{array}{ll}
k(k+1)r-1,      & \mbox{if } r \mbox{ is even};\\
(k^2+2k-1)r-1,  & \mbox{if both } k \mbox{ and } r \mbox{ are odd};\\
k(k+2)r-1,      & \mbox{if } k \mbox{ is even and } r \mbox{ is
odd},
\end{array}\right.$ \\
${\rm ~ ~  ~  ~ ~~~~~~\textit{and}}~f_r(k)\geq\left\{\begin{array}{ll}
\dfrac{k(k+1)r}{2}-1,      & \mbox{if } r \mbox{ is even};\\
\dfrac{k(k+1)r}{2},      & \mbox{if } r \mbox{ is odd}.
\end{array}\right.$
\end{theorem}

\begin{proof}
{\bf a)} We proceed by induction on $k$. Clearly $f_r(1)$ exists
(and, by Theorem~\ref{ThDi'}, we have $f_r(1)\leq 2r-1$). Suppose
that the assertion is true for $k-1$, where $k\geq 2$, and let us study
it for $k$. Let $G$ be an arbitrary graph of minimum degree at least
$$f_r(k-1)+\left\{\begin{array}{ll}
  2(k+1)r,  & \mbox{if } k \mbox{ is even and } r \mbox{ is odd};\\
  2kr       & \mbox{otherwise}.
\end{array}\right.$$
According to Theorem~\ref{ThSt}, there exists a partition $(V_1,V_2)$ of
$V(G)$ such that $G\lbrack V_1\rbrack$ is of minimum degree at least
$f_r(k-1)$ and $G\lbrack V_2\rbrack$ is of minimum degree at least
$2(k+1)r-1$ for $k$ even and $r$ odd, and of minimum degree at least
$2kr-1$ otherwise. Then $G\lbrack V_1\rbrack$ contains $k-1$
disjoint cycles $C_1,\ldots, C_{k-1}$ of different lengths all
divisible by $r$. By Corollary~\ref{CoThDi'}, subgraph $G\lbrack V_2\rbrack$
contains $k$ cycles of different lengths all divisible by $r$. Then
one of these cycles has length distinct from those of the cycles
$C_1,\ldots, C_{k-1}$. We get then a collection of $k$ disjoint
cycles of different lengths all divisible by $r$. So the assertion
is true for $k$, and the existence is proved.

\medskip

{\bf b)} If $r$ is even, then we have
\begin{align*}
f_r(k)  & \leq f_r(k-1)+2kr\\
        & \leq f_r(k-2)+2(k-1)r+2kr\\
        & \leq f_r(1)+\sum_{i=2}^k2ir\\
        & =k(k+1)r-1.
\end{align*}

If both $k$ and $r$ are odd, then we have
\begin{align*}
f_r(k)  & \leq f_r(k-1)+2kr\\
        & \leq f_r(k-2)+2kr+2kr\\
        & \leq f_r(k-3)+2(k-2)r+2kr+2kr\\
        & \leq f_r(1)+\sum_{i=1}^{(k-1)/2}2\cdot(4ir)\\
        & =(k^2+2k-1)r-1.
\end{align*}

If $k$ is even and $r$ is odd, then
\begin{align*}
f_r(k)  &\leq f_r(k-1)+2(k+1)r\\
        &\leq((k-1)^2+2(k-1)-1)r-1+2(k+1)r\\
        &=k(k+2)r-1.
\end{align*}
So the first assertion of \textbf{b)} is proved.

We now focus on the second assertion of \textbf{b)}. For even $r\geq 4$
and $k\geq 1$, the complete graph $K_{f_r(k)+1}$ contains $k$
vertex-disjoint cycles of different lengths, all divisible by $r$.
It follows that $f_r(k)+1\geq r+2r+\cdots +kr$, which yields $f_r(k)\geq
\dfrac{k(k+1)r}{2}-1$. For odd $r\geq 3$ and $k\geq 1$, the complete
bipartite graph $K_{f_r(k),f_r(k)}$ (of order $2f_r(k)$) contains
$k$ disjoint cycles of different lengths, all divisible by $r$. It
follows that $2f_r(k)\geq 2r+4r+\cdots +2kr$, which yields $f_r(k)\geq
\dfrac{k(k+1)r}{2}$. This concludes the proof.
\end{proof}
\medskip

We continue with the following result:

\begin{theorem}
{\bf a)} Let $r\geq 3$ be an odd integer. There exists a minimum
integer $\phi(r)$ such that every graph of minimum degree at least
$\phi(r)$ contains $r$ vertex-disjoint cycles $C_0,\ldots, C_{r-1}$ with
$v(C_i)\equiv i$ (mod $r$) for $0\leq i\leq r-1$.
\newline {\bf b)} We have $\dfrac{r^2+5r-2}{2}\leq\phi(r)\leq 2r^2-1$.
 \end{theorem}

\begin{proof}
\textbf{a)} Let $G$ be a graph with minimum degree at least
$2r^2-1$. By repeated applications of Theorem~\ref{ThSt}, we get a
partition $(V_0,V_1,\ldots,V_{r-1})$ of $V(G)$ such that for every
$i$, subgraph $G[V_i]$ has minimum degree at least $2r-1$. By Theorem~\ref{ThDi'}, 
subgraph $G[V_i]$ has a cycle $C_i$ of length $i$ (mod $r$).
Thus the collection of $r$ cycles $C_0,C_1,\ldots,C_{r-1}$ is as
required. So $\phi(r)$ exists, and $\phi(r)\leq
2r^2-1$.

\medskip

\textbf{b)} The complete graph $K_{\phi(r)+1}$ contains $r$ disjoint
cycles with the required conditions. It follows that $\phi(r)+1\geq
3+\cdots + r+2$, hence $\phi(r)+1\geq \dfrac{(r+2)(r+3)}{2}-3$,
which implies $\dfrac{r^2+5r-2}{2}\leq\phi(r)$, as claimed.
\end{proof}
\bigskip

We now study refinements of Theorem~\ref{ThMinimumDegree}.
Recall that this result implies that every graph of 
minimum degree at least $f(k)=\dfrac{k^2+5k-2}{2}$ has $k$
vertex-disjoint cycles of different lengths. Furthermore, the complete graph
$K_{f(k)+1}$ shows that the bound on the minimum degree is best
possible. However, if we allow finite exceptions, then this bound
may be improved, as indicated in the following result.

\begin{theorem}\label{ThLargeOrder}
{\bf a)} For every $k\geq 2$, every graph $G$ of order $n\geq
7\cdot\left\lfloor\dfrac{k^2}{4}\right\rfloor$ and $\delta(G)\geq
\dfrac{k(k+3)}{2}$ has $k$ vertex-disjoint cycles of different
lengths.\\
{\bf b)} The bound on $\delta(G)$ is best possible.
\end{theorem}

\begin{proof}
{\bf a)} We prove the claim by induction on~$k$, starting from the row $k=2$.
In other words, we prove that if $G$ is a graph of order $n\geq 7$ and $\delta(G) \geq 5$,
then $G$ contains two vertex-disjoint cycles of different lengths.

If $G$ is triangle-free, then we are done by Theorem
\ref{ThTriangleFree}. Suppose now that $G$ contains a triangle
$C=(x_1,x_2,x_3,x_1)$. We claim that there exists a vertex $x_4$ of $G-C$ having
three neighbors in $C$. Indeed, if it was not the case, then $G-C$ would be
of minimum degree at least $3$ and then it would contain two cycles
of different lengths. One of these cycles and $C$ would  then form a pair
of vertex-disjoint cycles of different lengths, and we would be
done. Clearly the subgraph $H$ induced by $x_1$, $x_2$, $x_3$ and
$x_4$ is complete. Suppose that there exists at most one vertex of
$G-H$ having four neighbors in $H$. Then the graph $G-H$ has at most
one vertex of degree at most $1$. Then $G-H$ contains a cycle (since
any longest path in $G-H$ has one of its end-vertices having at least two
neighbors on the path) and since $H$ contains a triangle and a
4-cycle, we get two vertex-disjoint cycles of different lengths and we are done. 
Suppose thus that there exist two vertices $x_5$
and $x_6$ both having $x_1$, $x_2$, $x_3$ and $x_4$ as neighbors.

Let $H'$ be the subgraph induced by $x_1,\ldots,x_6$. If every
vertex in $G-H'$ has at most three neighbors in $H'$, then
$\delta(G-H')\geq 2$ and $G-H'$ contains a cycle. Since $H\subset
H'$ contains a triangle and a 4-cycle, we get two vertex-disjoint
cycles of different lengths and we are done. So assume now that
there exists a vertex $x_7$ having at least four neighbors in $\lbrace
x_1,\ldots,x_6\rbrace$. Then $x_7$ has at least two neighbors in
$\lbrace x_1,\ldots,x_4\rbrace$, say $x_7x_1,x_7x_2\in E(G)$. Then
$(x_1,x_2,x_7,x_1)$ and $(x_3,x_5,x_4,x_6,x_3)$ are two
vertex-disjoint cycles of different lengths. This achieves the
proof of the case $k=2$.

\medskip

Suppose now that the assertion is true up to row
$k-1$, where $k\geq 3$, and let us study it for $k$. So $G$ is a graph of
order $n\geq 7\cdot\left\lfloor\dfrac{k^2}{4}\right\rfloor$ and of
minimum degree
$\delta(G)\geq\dfrac{k(k+3)}{2}=\dfrac{(k+1)(k+2)}{2}-1$.
%Clearly we may suppose that $\delta(G)=\dfrac{k(k+3)}{2}$.
By applying Theorem \ref{ThSt} repeatedly, we get a partition
$(V_1,V_2,\ldots,V_t)$ of $V(G)$ where $t=\left\lfloor\dfrac{k+1}{2}\right\rfloor$,
such that each $V_i$ induces a graph of minimum degree at least
$k+1$. Hence each $G[V_i]$ has a $k$-path-vertex schema. This
implies that $G$ has a $k$-path-vertex schema of order at most
$\dfrac{n}{\lfloor(k+1)/2\rfloor}$.

Suppose first that there exists an optimal $k$-path-vertex schema
$\mathcal{S}=(P,x)$ such that every vertex $y$ of $G-\mathcal{S}$
has at most $k+1$ neighbors in $\mathcal{S}$. Recall that
$|V(\mathcal{S})|\leq\dfrac{n}{\lfloor(k+1)/2\rfloor}$. Let
$G'=G-\mathcal{S}$. Then the order of $G'$ is
$$n'=n-|V(\mathcal{S})|\geq\dfrac{\lfloor(k-1)/2\rfloor}{\lfloor(k+1)/2\rfloor}n
\geq\left\{\begin{array}{ll} 7k(k-2)/4, & k \mbox{ is even};\\
7(k-1)^2/4, & k \mbox{ is odd}, \end{array}\right.$$ \textit{i.e.} $n'\geq
7\cdot\left\lfloor\dfrac{(k-1)^2}{4}\right\rfloor$. Note that every
vertex $y$ of $G'$ has at least
$$\dfrac{k^2+3k}{2}-k-1=\dfrac{(k-1)^2+3(k-1)}{2}$$ neighbors in $G'$,
so $\delta(G')\geq\dfrac{(k-1)^2+3(k-1)}{2}$. By the induction
hypothesis $G'$ contains $k-1$ vertex-disjoint cycles $C_1,\dots,
C_{k-1}$ of different lengths. Plus, the induced subgraph $G\lbrack
V(\mathcal{S})\rbrack$ contains $k$ cycles of different lengths. Then
one of these cycles and the cycles $C_1,\dots, C_{k-1}$ form a
collection of $k$ vertex-disjoint cycles of different lengths. So in
this case the assertion is proved for $k$.

Suppose now that there does not exist an optimal $k$-path-vertex
schema $\mathcal{S}$ such that every vertex $y$ of $G-\mathcal{S}$
has at most $k+1$ neighbors in $V(\mathcal{S})$. We take an
arbitrary optimal $k$-path-vertex schema $\mathcal{S}=(P,x)$, and
let $y\in V(G-\mathcal{S})$ be a vertex having $k+2$
neighbors in $\mathcal{S}$. By Lemma \ref{LeOptimalSchema}, vertices $x$ and
$y$ are adjacent, all the vertices of $P$ are neighbors of $x$ (and
also of $y$), and $|V(\mathcal{S})|=k+2$. We put
$P=(u_1,u_2,\ldots,u_{k+1})$ and $\varOmega=\lbrace
u_1,u_2\ldots,u_{k-1},x,y\rbrace$. Clearly $G\lbrack
\varOmega\rbrack$ contains $k-1$ cycles of different lengths (ranging from
$3$ to $k+1$).

Let $G'=G-\varOmega$. It is easy to see that $G'$ has minimum degree
at least $\dfrac{(k-1)^2+3(k-1)}{2}$ and has order at least
$7\cdot\left\lfloor\dfrac{(k-1)^2}{4}\right\rfloor$. By the induction
hypothesis $G'$ contains $k-1$ disjoint cycles
$C_1,C_2,\ldots,C_{k-1}$ of different lengths (where we label the $C_i$'s so that their lengths are non-decreasing).
Since $G\lbrack \varOmega\rbrack$ contains $k-1$ cycles of all
lengths from $3$ to $k+1$, it is easy to see that $\lvert
V(C_i)\rvert=i+2$ for $1\leq i\leq k-1$. In particular, cycle $C_{k-1}$ is
of order $k+1$ and $C_{k-2}$ is of order $k$.

Let
$G''=G-\left(\varOmega\cup\bigcup\limits_{i=1}^{k-1}V(C_i)\right)$,
and let $z$ be a vertex of $G''$. We claim that $z$ has at most one
neighbor in $\varOmega$. Indeed, otherwise it is easy to see
that $z$ and the vertices of $\varOmega$ would form a cycle of order $k+2$;
so, together with $C_1, ..., C_{k-2}$ we would then get $k$
disjoint cycles of different lengths. We also claim that $z$ has at most
$\dfrac{k+1}{2}$ neighbors in $C_{k-1}$. Indeed, otherwise the
vertices of $C_{k-1}$ and $z$ would form a cycle $C'_{k-1}$ of order
at least $k+2$; so $C_1,\ldots,C_{k-2},C'_{k-1}$ and an
appropriate cycle of $G\lbrack \varOmega\rbrack$ of order $k+1$
would form a collection of $k$ disjoint cycles of different lengths.
It follows that $z$ has at most $1+3+\cdots+k+\dfrac{k+1}{2}=\dfrac
{k^2+2k-3}{2}$ neighbors in
$\varOmega\cup\bigcup\limits_{i=1}^{k-1}V(C_i)$, and $z$ has at
least $\dfrac{k+3}{2}\geq 3$ neighbors in $G''$. Let
$P'=(v_1,v_2,\ldots,v_s)$ be a longest path in $G''$. Since
$\delta(G'')\geq 3$, we have $s\geq 4$.

Note that all the neighbors of $v_1$ in $G''$ are in $P'$. It
follows that $v_1$ has at most $k$ neighbors in $G''$ (for otherwise
we would have a cycle of length at least $k+2$ and, together with the
cycles $C_1, \ldots C_{k-1}$, we would have the desired cycles). We have seen that
$v_1$ has at most one neighbor in $\varOmega$, and that $v_1$ has at
most $\dfrac{k+1}{2}$ neighbors in $C_{k-1}$. We claim that $v_1$
has more than $\dfrac{k}{2}$ neighbors in $C_{k-2}$. Indeed, suppose
the opposite. Then $v_1$ has at most $1+3+\cdots
+k-1+\dfrac{k}{2}+\dfrac{k+1}{2}=\dfrac{k^2+k-3}{2}$ neighbors in
$\varOmega\cup\bigcup\limits_{i=1}^{k-1}V(C_i)$. This means that $v_1$
has at most $\dfrac{k^2+k-4}{2}$ neighbors in
$\varOmega\cup\bigcup\limits_{i=1}^{k-1}V(C_i)$ and that $v_1$ has
at least $\dfrac{k^2+3k}{2}-\dfrac{k^2+k-4}{2}=k+2$ neighbors in
$G''$, a contradiction. So as we claimed, vertex $v_1$ has more than
$\dfrac{k}{2}$ neighbors in $C_{k-2}$, and similarly $v_s$ has more
than $\dfrac{k}{2}$ neighbors in $C_{k-2}$. But now it is easy to
see that the vertices of $C_{k-2}$ and the vertices of $P'$ form a
cycle of length at least $k+2$, and, together with the cycles $C_i$ 
(where $1\leq i\leq k-1$ and $i\neq k-2$) and an appropriate cycle of
length $k$ of $G\lbrack \Omega\rbrack$, we get $k$ disjoint cycles
of different lengths. This concludes this part of the proof.

\medskip

\textbf{b)} Let $G$ be a complete bipartite graph with two partite
sets of size $\dfrac{k(k+3)}{2}-1$ and $n-\dfrac{k(k+3)}{2}+1$,
respectively, where $n$ is large enough. Then
$\delta(G)=\dfrac{k(k+3)}{2}-1$. Note that $G$ does not have $k$
vertex-disjoint cycles of different lengths. So the bound
on $\delta(G)$ in our statement is sharp.
\end{proof}
\medskip

If we now consider graphs with bounded maximum degree, then we may
further refine Theorem \ref{ThLargeOrder}. 
In particular, we believe the following conjecture should be true.

\begin{conjecture}\label{ConjMaximumDegree}
For every two integers $k$ and $D$, there is an integer $n_0$ such that every
graph $G$ of order at least $n_0$ with
$k+1\leq\delta(G)\leq\Delta(G)\leq D$ has $k$ vertex-disjoint
cycles of different lengths.
\end{conjecture}

To conclude this section, we show that Conjecture~\ref{ConjMaximumDegree} is true for $k=2$.

\begin{theorem}\label{ThConjMaximumDegree}
For every $D\geq 3$, every graph $G$ of order more than
$20D-46$ with $3 \leq \delta(G) \leq \Delta(G) \leq D$ has two vertex-disjoint cycles of different lengths.
\end{theorem}

\begin{proof}
Suppose that $G$ does not contain two
vertex-disjoint cycles of different lengths. We will show that
$n(G)$ is bounded above by $20D-46$. Let $g$ denote the girth of $G$. We
distinguish two cases.

\medskip

\noindent\textbf{Case 1.} $g=3$ or $g=4$.

Let $C$ be a shortest cycle of $G$. Then $C$ is either a $C_3$ or a
$C_4$. Let $H$ be an induced subgraph of $G$ defined by a sequence
of vertex sets $U_0,U_1,\cdots,U_s$ such that:

\begin{enumerate}
	\item $U_0=V(C)$ and $U_s=V(H)$;
	
	\item for every $i$, where $0\leq i\leq s-1$, there is a vertex $x_i\in
V(G)\backslash U_i$ such that $d_{U_i}(x_i)>\dfrac{d(x_i)}{2}$ and
$U_{i+1}=U_i\cup\{x_i\}$;

	\item for every vertex $x\in V(G)\backslash U_s$, we have $d_{U_s}(x)\leq
\dfrac{d(x)}{2}$.
\end{enumerate}

Note that $n(H)=n(C)+s$. We first show that $s$ is bounded above by a
constant (depending on $D$). Note that $e(U_0,V(G)\backslash
U_0)=\sum_{v\in U_0}(d(v)-2)$ is bounded above by $4D-8$. It is easy to
see that $e(U_{i+1},V(G)\backslash U_{i+1})\leq e(U_i,V(G)\backslash
U_i)-1$. This implies that $e(U_s,V(G)\backslash U_s)\leq 4D-8-s$
and $s\leq 4D-8$.

Let $F=G-H$. Recall that for every vertex $x\in V(F)$, we have
$d_F(x)\geq\left\lceil \dfrac{d(x)}{2}\right\rceil\geq 2$. Thus
$\delta(F)\geq 2$. If $F$ contains a cycle of length different from that of 
$C$, then we are done. So we assume that either every cycle
of $F$ is a triangle, or every cycle of $F$ is a $C_4$. Note that
every vertex in $N_2(F)$ has a neighbor in $H$. Furthermore, we have $n_2(F)\leq
e(V(H),V(F))\leq 4D-8-s$. Now, applying Lemmas~\ref{LeCycleTriangle} and~\ref{LeCycleQuadrangle}, 
we get $$n(F)\leq 5(n_2(F)-2)\leq 5(4D-10-s),$$ and,
thus,
$$n(G)=n(H)+n(F)\leq 4+s+5(4D-10-s)\leq 20D-46.$$

\medskip

\noindent\textbf{Case 2.} $g\geq 5$.

Let $C=(x_1,\ldots,x_g,x_1)$ be a cycle of $G$ of order $g$. It is easy to see that any vertex of $V(G)\setminus V(C)$ has at most one neighbor in $C$ (for otherwise we would have a cycle of length less than $g$). So, the graph $G'=G-C$ is of minimum degree at least $2$ and since $G'$ cannot be of degree at least $3$ (for otherwise $G'$ would contain two cycles of different lengths, and then one of these cycles and $C$ would be disjoint and of different length), subgraph $G'$ is of minimum degree $2$. Let $P=(y_1,\ldots, y_t)$ be a longest path in $G'$. By maximality of $t$, all neighbors of $y_1$ in $G'$ are in $P$. If $y_1$ has three neighbors in $G'$, clearly we are done. It follows that $y_1$ has exactly two neighbors in $P$ and then, necessarily, the neighbor of $y_1$ in $P$ distinct from $y_2$ is $y_g$. So $C'=(y_1,\ldots, y_g, y_1)$ is a cycle of $G'$ of length $g$. We claim that $d_{G'}(y_2) = 2$. Indeed suppose that $y_2$ has a neighbor in $P$ distinct from $y_1$ and $y_3$. Necessarily this neighbor is $y_{g+1}$. But then $(y_1,y_g,y_{g+1},y_2,y_1)$ is a $4$-cycle, which is impossible. Suppose now that $y_2$ has a neighbor $y$ in $V(G')\setminus V(P)$. Then $(y, y_2,\ldots, y_t)$ is a longest path in $G'$. It follows that $y_g$ is a neighbor of $y$. But then $(y_1,y_2,y,y_g,y_1)$ is a $4$-cycle, which is again impossible. Each of the vertices $y_1$ and $y_2$ has one neighbor in $C$. It is then easy to see that when $g \geq 7$, vertices $y_1$, $y_2$ and some vertices of $C$ form a cycle of length less than $g$, which is impossible. So we are done when $g \geq 7$.

Assume first $g=5$. Since each of the vertices $y_1$, $y_2$ has one neighbor in $C$,  vertices $y_1$, $y_2$ and some vertices of $C$ form a cycle $C_1$ of length $6$ (since $g=5$).
Let us consider the graph $G_1=G-(V(C)\cup \lbrace y_1, y_2\rbrace$. Clearly every vertex of $G_1$ distinct from $y_3$ and $y_5$ is of degree at least $2$ in $G_1$. If one of the vertices $y_3$ and $y_5$ is of degree at least $2$ in $G_1$,  then, by considering a longest path in $G_1$, it is easy to see that $G_1$ contains a cycle, and then this cycle and one of the cycles $C$ and $C_1$ would form two vertex-disjoint cycles of different lengths. Suppose now that  $y_3$ and $y_5$ are both of degree $1$ in $G_1$. Since $(y_4, y_3, y_2,y_1,y_5,\ldots, y_t)$ is a path of length $t$, by the previous arguments $y_4$ is of degree $2$ in $G'$. It is easy to see that the graph $G_2=G-V(C)\cup V(C')$ is of minimum degree at least $2$. It follows that  $G_2$ contains a cycle, and then, as previously, we are done. This concludes the case $g=5$.

%If $g=5$, then the neighbors of $y_1$ and $y_2$ in $C$ are non-adjacent (for otherwise we would find a cycle of length less then~5). Suppose $y_1x_1,y_2x_3\in E(G)$ without loss of generality. Since $(y_4,y_3,y_2,y_1,y_5,\ldots,y_t)$ is a longest path in $G-C$, we get that $y_3$ and $y_4$ have neighbors in $C$. It follows that $y_3$ neighbors $x_5$, and $y_4$ neighbors $x_2$ (for otherwise we would find a cycle of length less then 5). Recall that $G[V(C)\cup V(C')]$ (which is of order 10) contains two cycles of different lengths. So there are some vertices in $G''=G-C-C'$; let thus $P'=(z_1,z_2,\ldots,z_s)$ be a longest path in $G-C-C'$. Recall that $G''$ has no cycle and every vertex in $G'$ has at most one neighbor in $C$ and at most one neighbor in $C'$. It follows that $z_1$ (and $z_s$) has one neighbor in $C$, one neighbor in $C'$, and one neighbor in $P'$. If $z_1$ and $z_s$ neighbor a common vertex in $C$, say $x_i$, then $G[V(P)\cup\{x_i\}]$ contains a cycle and $G[(V(C)\cup V(C'))\backslash\{x_i\}]$ has two cycles of different lengths -- and hence we are done. If $z_1$ and $z_s$ neighbor two different vertices in $C$, say $x_i$ and $x_j$, respectively, then $x_i$ and $x_j$ divide $C$ into two paths of different lengths and $G[V(P)\cup V(C)$ contains two cycles of different lengths. Since one of these two cycles has length different from the length of $C'$, we are done.	

Suppose now that $g=6$. Let $P'=(y_5,\ldots, y_1, y_6,\ldots, y_t)$ be a longest path of $G'$. By the previous reasoning, vertices $y_5$ and $y_4$ are of degree $2$ in $G'$. It follows that each of the vertices $y_1$, $y_2$ and $y_4$ has one neighbor in $C$. Then it is easy to see that the subgraph $G_1$ of $G$ induced by $V(C)\cup\lbrace y_1, y_2, y_3, y_4 \rbrace$ contains two cycles of different lengths.
It is clear that $C''=(y_t, y_{t-1},\ldots, y_{t-5}, y_t)$ is a $6$-cycle of $G$. If $t\geq 10$, clearly $V(C'')$ is vertex-disjoint with $V(G_1)$; and then it is easy to see that we are done. It cannot be the case that $t=9$ or $t=7$ as $x_4$ and $x_2$ are of degree $2$ in $G'$. Suppose now that $t=6$. It is then easy to see that the graph $G-V(C)\cup V(P)$ does not contain cycles (for otherwise we are done). So $z$ has a neighbor $y_i$ (where $i=3$ or $i=6$). But then the vertices of $C$ and $z$ would form a path of order $7$, a contradiction.

Suppose now that $t=8$. Clearly, vertices $y_7$ and $y_8$ are of degree $2$ in $G'$. Since $n\geq 15$, subset $V(G)\setminus (V(C)\cup V(P))$ is non-empty, and the graph $G-V(C)\cup V(P)$ is of minimum degree at most $1$. Then there exists a vertex $z$ of $G-V(C)\cup V(P)$ having a neighbor in $V(P)$. Since $y_1$, $y_2$, $y_4$, $y_5$, $y_7$ and $y_8$ are of degree $2$ in $G'$, this neighbor is either $y_3$ or $y_6$. Without loss of generality, we may suppose that this neighbor is $y_3$. Suppose that $z$ has another neighbor in $P$. Necessarily this other neighbor is $y_6$ -- but then $(z,y_3,y_4,y_5,y_6,z)$ would be a $5$-cycle, which is impossible. So the only neighbor of $z$ in $P$ is $y_3$. Then $z$ has a neighbor $u$ in $G-V(C)\cup V(P)$. Suppose that $u$ has a neighbor $v$ in $G-V(C)\cup V(P)$. Then $(v,u,z,y_3,\ldots,y_8)$ would be a path of $G'$ of length $9$, which by maximality of $t=8$ is impossible. It follows that $u$ has a neighbor in $P$, and necessarily this neighbor is $y_6$ (because $G'$ does not contain triangles). It is easy to see that $z$ and $u$ are of degree $2$ in $G'$. So the eight vertices of $\varOmega=\lbrace y_1,y_2,y_4,y_5,y_7,y_8, z,u\rbrace$ are of degree $2$ in $G'$ and then each vertex of $\varOmega$ has exactly one neighbor in $C$ (which has order $6$). It follows that there exists a vertex $w$ of $C$ having two neighbors $a$ and $b$ in $\varOmega$. It is easy to see that $d_{G'}(a,b)\leq 3$ and then we get a cycle of length at most $5$, which is not possible. So the case $g=6$ is settled and the result is proved for $g\geq 5$. This concludes the whole proof.
\end{proof}
\medskip

Taking $D=3$ in Theorem \ref{ThConjMaximumDegree}, we get that every
cubic graph of order more than $20D-46=14$ has two vertex-disjoint
cycles of different lengths.

\begin{theorem}\label{ThCubicGraph}
Every cubic graph of order at least $15$ contains two
vertex-disjoint cycles of different lengths.
\end{theorem}

\begin{figure}[t]
\begin{center}
\begin{tikzpicture}
   \tikzstyle{VertexStyle}=[shape = circle, fill=black, minimum size = 4pt, inner sep=0pt]
   \SetVertexNoLabel	
   \tikzstyle{EdgeStyle}=[line width=1pt]
   \grHeawood[RA=2.5]
\end{tikzpicture}
\end{center}
\caption{The Heawood graph.}
\label{figure:heawood}
\end{figure}
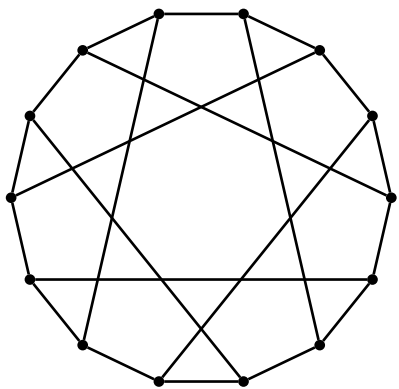

\noindent We remark that the Heawood graph
($(3,6)$-cage, see Figure~\ref{figure:heawood}) of order 14 does not contain two vertex-disjoint
cycles of different lengths. Thus the bound on the
order in Theorem~\ref{ThCubicGraph} is sharp.

%%%%%%%%%%%%%%%%%%%%%%%%%%%%%%%%%%%%%%%%%%%%%%%%%%%%%%%%%%%%
%%%%%%%%%%%%%%%%%%%%%%%%%%%%%%%%%%%%%%%%%%%%%%%%%%%%%%%%%%%%
%%%%%%%%%%%%%%%%%%%%%%%%%%%%%%%%%%%%%%%%%%%%%%%%%%%%%%%%%%%%
%%%%%%%%%%%%%%%%%%%%%%%%%%%%%%%%%%%%%%%%%%%%%%%%%%%%%%%%%%%%
%%%%%%%%%%%%%%%%%%%%%%%%%%%%%%%%%%%%%%%%%%%%%%%%%%%%%%%%%%%%
%%%%%%%%%%%%%%%%%%%%%%%%%%%%%%%%%%%%%%%%%%%%%%%%%%%%%%%%%%%%
%%%%%%%%%%%%%%%%%%%%%%%%%%%%%%%%%%%%%%%%%%%%%%%%%%%%%%%%%%%%
%%%%%%%%%%%%%%%%%%%%%%%%%%%%%%%%%%%%%%%%%%%%%%%%%%%%%%%%%%%%
%%%%%%%%%%%%%%%%%%%%%%%%%%%%%%%%%%%%%%%%%%%%%%%%%%%%%%%%%%%%
%%%%%%%%%%%%%%%%%%%%%%%%%%%%%%%%%%%%%%%%%%%%%%%%%%%%%%%%%%%%
%%%%%%%%%%%%%%%%%%%%%%%%%%%%%%%%%%%%%%%%%%%%%%%%%%%%%%%%%%%%
%%%%%%%%%%%%%%%%%%%%%%%%%%%%%%%%%%%%%%%%%%%%%%%%%%%%%%%%%%%%
%%%%%%%%%%%%%%%%%%%%%%%%%%%%%%%%%%%%%%%%%%%%%%%%%%%%%%%%%%%%
%%%%%%%%%%%%%%%%%%%%%%%%%%%%%%%%%%%%%%%%%%%%%%%%%%%%%%%%%%%%
%%%%%%%%%%%%%%%%%%%%%%%%%%%%%%%%%%%%%%%%%%%%%%%%%%%%%%%%%%%%
%%%%%%%%%%%%%%%%%%%%%%%%%%%%%%%%%%%%%%%%%%%%%%%%%%%%%%%%%%%%
%%%%%%%%%%%%%%%%%%%%%%%%%%%%%%%%%%%%%%%%%%%%%%%%%%%%%%%%%%%%
%%%%%%%%%%%%%%%%%%%%%%%%%%%%%%%%%%%%%%%%%%%%%%%%%%%%%%%%%%%%
%%%%%%%%%%%%%%%%%%%%%%%%%%%%%%%%%%%%%%%%%%%%%%%%%%%%%%%%%%%%
%%%%%%%%%%%%%%%%%%%%%%%%%%%%%%%%%%%%%%%%%%%%%%%%%%%%%%%%%%%%
%%%%%%%%%%%%%%%%%%%%%%%%%%%%%%%%%%%%%%%%%%%%%%%%%%%%%%%%%%%%
%%%%%%%%%%%%%%%%%%%%%%%%%%%%%%%%%%%%%%%%%%%%%%%%%%%%%%%%%%%%
%%%%%%%%%%%%%%%%%%%%%%%%%%%%%%%%%%%%%%%%%%%%%%%%%%%%%%%%%%%%

\section{Disjoint cycles of different lengths in digraphs} \label{section:directed}

The results from this section are motivated by the following conjecture raised by Bermond and Thomassen in 1981~\cite{BT81}.

\begin{conjecture}[Bermond and Thomassen~\cite{BT81}] \label{conjecture:bt}
For every $k \geq 2$, every digraph of minimum out-degree at least $2k-1$ contains at least $k$ vertex-disjoint directed cycles.
\end{conjecture}

\noindent If true, Conjecture~\ref{conjecture:bt} would be best possible because of the bidirected complete graph on $2k$ vertices.
Towards Conjecture~\ref{conjecture:bt}, one could more generally wonder whether for every $k$ there is a smallest finite function $f(k)$ such that every digraph of minimum out-degree at least $f(k)$ contains at least $k$ vertex-disjoint directed cycles (so $f(k)$ should be equal to $2k-1$ according to Conjecture~\ref{conjecture:bt}). Thomassen first proved in~\cite{Tho83} that $f(k)$ exists for every $k \geq 1$. Later, Alon improved the value of $f(k)$ to $64k$~\cite{Alo96}. For $k=2$, Thomassen proved in~\cite{Tho83} that $f(k)=3$, which confirms Conjecture~\ref{conjecture:bt} for $k=2$. Later on, Lichiardopol, P\'or and Sereni proved that for $k=3$ the best value for $f(k)$ is $5$, again confirming Conjecture~\ref{conjecture:bt} for $k=3$~\cite{LPS09}. This apart, Conjecture~\ref{conjecture:bt} is still open, though some more partial results may be found in literature (see \textit{e.g.} \cite{BLS09}).

Motivated by Conjecture~\ref{conjecture:bt} and in the same flavour as in Section~\ref{section:undirectedg}, one can wonder about the existence of a (smallest) finite function $g(k)$ such that every digraph with minimum out-degree at least~$g(k)$ contains $k$ (vertex-) disjoint (directed) cycles of distinct lengths. In this context, the following was conjectured by Lichiardopol~\cite{Lic14}:

\begin{conjecture}[Lichiardopol~\cite{Lic14}] \label{conjecture:digraph-lengths}
For every $k \geq 2$, there exists an integer $g(k)$ such that every digraph of minimum out-degree at least $g(k)$ contains $k$ vertex-disjoint directed cycles of distinct lengths.
\end{conjecture}

\noindent It is worth pointing out that a similar function $h(k)$ for the existence of $k$ disjoint cycles of the same length does not exist, 
as Alon proved that there exist digraphs of arbitrarily large minimum out-degree having no two (not necessarily disjoint) cycles of the same length~\cite{Alo96}.
Lichiardopol proved Conjecture~\ref{conjecture:digraph-lengths} for $k=2$ in~\cite{Lic14}, solving a question of Henning and Yeo~\cite{HY12}. We note that $g(k)$ should in general be a quadratic function of $k$; for an illustration of this statement, consider a complete bidirected digraph $D$ on $g(k)+1$ vertices. Since $\delta^+(D)=g(k)$, there exist $k$ disjoint cycles of different lengths in $D$. It follows that $$g(k)+1\geq 2+\cdots +k+1,$$ hence $g(k)+1\geq \frac{(k+1)(k+2)}{2}-1$, which yields $g(k)\geq \frac{k^2+3k-2}{2}$ in general.

%Let $D$ be a digraph whose vertex set is a partition $V_0,V_1,V_2$, each $V_i$ containing $k$ independent vertices, and in which every vertex of any part $V_i$ has an out-going arc towards every vertex of $V_{i+1}$ (see Figure~\ref{figure:r-square}). Note that every vertex of $D$ has out-degree exactly~$k$. Since a cycle of $D$ has to go from $V_0$ to $V_1$, then from $V_1$ to $V_2$, then to $V_2$ to $V_0$, and so on in this order, clearly to get as many distinct cycles of $D$ as possible, one has to pick a cycle of length~$3$ (using one vertex from each $V_i$), then a cycle of length~$6$ (using two vertices from each $V_i$), and so on. So, assuming we can pick $x$ disjoint cycles in $D$, clearly $k \geq \frac{x(x+1)}{2}$.

\medskip

In this section, we verify Conjecture~\ref{conjecture:digraph-lengths} in several contexts. 
We first verify it for tournaments in Section~\ref{section:tournaments}.
Using the probabilistic method, we then verify it, in Sections~\ref{section:regular} and~\ref{section:small},
when $D$ is regular or when its order is a polynomial function of its minimum out-degree, respectively.
Some concluding remarks are gathered in Section~\ref{section:ccl-digraphs}.

%%%%%%%%%%%%%%%%%%%%%%%%%%%%%%%%%%%%%%%%%%%%%%%%%%%%
%%%%%%%%%%%%%%%%%%%%%%%%%%%%%%%%%%%%%%%%%%%%%%%%%%%%
%%%%%%%%%%%%%%%%%%%%%%%%%%%%%%%%%%%%%%%%%%%%%%%%%%%%
%%%%%%%%%%%%%%%%%%%%%%%%%%%%%%%%%%%%%%%%%%%%%%%%%%%%
%%%%%%%%%%%%%%%%%%%%%%%%%%%%%%%%%%%%%%%%%%%%%%%%%%%%
%%%%%%%%%%%%%%%%%%%%%%%%%%%%%%%%%%%%%%%%%%%%%%%%%%%%

\subsection{Tournaments} \label{section:tournaments}

We here verify Conjecture~\ref{conjecture:digraph-lengths} for tournaments. 
More precisely, for every $k \geq 1$, we study the smallest finite function $g_t(k)$ such that
every tournament of minimum out-degree at least $g_t(k)$ has $k$ vertex-disjoint directed cycles of different lengths.
We exhibit both an upper bound and a lower bound on $g_t(k)$ for every $k$.

In order to prove our upper bound, we need to introduce the following result first.

\begin{lemma} \label{lemma:tournaments}
Every tournament of minimum out-degree $\delta\geq 1$ contains a directed cycle of order at least $2\delta+1$. 
\end{lemma}

\begin{proof}
We proceed by induction on the order $n\geq 2\delta+1$ of a tournament $T$ of minimum out-degree $\delta$. We claim that the assertion is true for $n=2\delta+1$. Indeed, in this case $T$ is a regular tournament of degree $\delta$. So $T$ is strong, and then, by Camion's Theorem, we get that $T$ is Hamiltonian. So the vertices of $T$ form a directed cycle of order $2\delta+1$, and we are done. Suppose now that the assertion is true up to the row $n-1$, where $n\geq 2\delta+2$, and let us study it for $n$. So $T$ is a tournament of minimum out-degree $\delta$ and of order $n\geq 2\delta+2$. If $T$ is strong, then $T$ is Hamiltonian and again we are done. Suppose thus that $T$ is not strong. Then there exists a partition $(A,B)$ of $V(T)$ such that $A$ dominates $B$ (that is every ordered pair $(x,y)$ with $x\in A$ and $y\in B$ is an arc of $T$). Clearly the induced tournament $T\lbrack B\rbrack$ is of minimum out-degree at least $\delta$. So, by the induction hypothesis, we deduce that $T\lbrack B\rbrack$ (and therefore $T$) contains a directed cycle of order at least $2\delta+1$. So the assertion is true for $n$, which concludes the proof.
\end{proof}
\medskip

We are now ready to exhibit an upper bound on $g_t(k)$, and hence to confirm Conjecture~\ref{conjecture:bt} for tournaments.

\begin{theorem}  \label{theorem:tournament}
For every $k\geq 1$, we have $g_t(k) \leq \frac{k^2+4k-3}{2}$. 
\end{theorem}

\begin{proof}
We proceed by induction on $k$. The assertion is clearly true for $k=1$. So suppose that the assertion is true up to row $k-1$ (where $k\geq 2$), and let us study it for $k$. Let $T$ be a tournament of minimum out-degree at least $\frac{k^2+4k-3}{2}$. By the induction hypothesis $T$ contains $k-1$ disjoint cycles $C_1,\ldots, C_{k-1}$ of different lengths. It is easy to see that for every $1\leq i\leq k-1$, cycle $C_i$ contains a cycle $C'_i$ of length $i+2$. We get then a collection $C_1', ..., C_{k-1}'$ of disjoint cycles with $\lvert V(C_i')\rvert =i+2$, and therefore of different lengths.
 
We have $$\lvert V(C_1')\cup\cdots\cup V(C_{k-1}')\rvert =3+\cdots + k+1 =\frac{k^2+3k-4}{2}<\frac{k^2+4k-3}{2}.$$ It thus follows that the tournament $T'=T-V(C_1')\cup\cdots\cup V(C_{k-1}')$ is of positive order and of minimum out-degree at least $\frac{k^2+4k-3}{2}-\frac{k^2+3k-4}{2}=\frac{k+1}{2}$. According to Lemma~\ref{lemma:tournaments}, we know that $T'$ contains a cycle $C'_k$ of length at least $k+2$ and, together with the directed cycles $C_1', ..., C_{k-1}'$, we get then $k$ disjoint cycles of different lengths. Consequently, the assertion is true for $k$, and the result is proved. 
\end{proof}
\medskip

We now deduce a lower bound on $g_t(k)$.

\begin{observation} \label{observation:tournaments}
For every $k \geq 1$, we have $g_t(k)\geq \frac{k^2+5k-2}{4}$.
\end{observation}

\begin{proof}
Let $T$ be a regular tournament of degree $g_t(k)$. 
Then $T$ contains $k$ disjoint cycles of different lengths. 
Since the order of $T$ is $2g_t(k)+1$, it follows that $$2g_t(k)+1\geq 3+\cdots +k+2,$$ 
hence $2g_t(k)+1\geq \frac{(k+2)(k+3)}{2}-3$, which yields $g_t(k)\geq \frac{k^2+5k-2}{4}$.
\end{proof}
\medskip

So, in the context of tournaments, according to Theorem~\ref{theorem:tournament} and Observation~\ref{observation:tournaments} we get the following.

\begin{corollary} \label{corollary:tournaments}
For every $k \geq 1$, we have $$\frac{k^2+5k-2}{4} \leq g_t(k) \leq \frac{k^2+4k-3}{2}.$$
\end{corollary}

%%%%%%%%%%%%%%%%%%%%%%%%%%%%%%%%%%%%%%%%%%%%%%%%%%%%
%%%%%%%%%%%%%%%%%%%%%%%%%%%%%%%%%%%%%%%%%%%%%%%%%%%%
%%%%%%%%%%%%%%%%%%%%%%%%%%%%%%%%%%%%%%%%%%%%%%%%%%%%
%%%%%%%%%%%%%%%%%%%%%%%%%%%%%%%%%%%%%%%%%%%%%%%%%%%%
%%%%%%%%%%%%%%%%%%%%%%%%%%%%%%%%%%%%%%%%%%%%%%%%%%%%
%%%%%%%%%%%%%%%%%%%%%%%%%%%%%%%%%%%%%%%%%%%%%%%%%%%%

\subsection{Regular digraphs} \label{section:regular}

We now use the probabilistic method to prove, in the current section and upcoming Section~\ref{section:small}, 
Conjecture~\ref{conjecture:digraph-lengths} in two new contexts.
To this aim, we first need to introduce a few tools and lemmas.
The first two are classic tools of the probabilistic method, namely Chernoff's Inequality and the Lov\'asz Local Lemma (see \textit{e.g.}~\cite{MR02}).

\begin{proposition}[Chernoff's Inequality] \label{prop: Chernoff}
Let $X$ be a binomial random variable $BIN(n,p)$. Then, for any $0 \leq t \leq np$,
we have $Pr[|X-np|> t] \leq 2 e^{-t^2/3np}$.
\end{proposition}

\begin{proposition}[Lov\'{a}sz Local Lemma -- Symmetric version] \label{prop: Lovasz}
Let $A_1,..., A_n$ be a finite set of events in some probability space $\Omega$ such that each $A_i$ occurs with probability at most $p$, where each $A_i$ is mutually independent of all but at most $d$ other events. If $4pd \leq 1$, then $Pr[\cap_{i=1}^{n} \overline{A_i}] > 0$.
\end{proposition}

We will also be needing the following fact on the existence of $k$ (not necessarily disjoint) cycles with distinct lengths in a digraph with minimum out-degree at least~$k$.

\begin{proposition} \label{lemma:k-cycle-lengths}
For every $k \geq 1$, every digraph $D$ with minimum out-degree at least $k$ contains $k$ directed cycles of distinct lengths.
\end{proposition}

\begin{proof}
Consider the out-neighbours of the last vertex of a longest directed path in $D$.
\end{proof}
\bigskip

We are now ready to prove the main result of this section. By a \textit{$r$-regular digraph}, we refer to a digraph whose all vertices have in- and out-degree $r$.

\begin{theorem} \label{theorem:regular}
Let $k \geq 1$ and $r \geq \frac{k^2}{2}(1+7(\frac{\ln k}{k})^\frac{1}{3})$. Then every $r$-regular digraph contains at least $k$ vertex-disjoint directed cycles of distinct lengths.
\end{theorem}

\begin{proof}
Let $D$ be a simple $r$-regular digraph, and assume $r=\frac{k^2}{2}(1+7(\frac{\ln k}{k})^\frac{1}{3})$. The proof reads as follows. The main idea is to prove, using several probabilistic tools, that we can partition the vertex set of $D$ into $k$~parts $V_1,...,V_k$ such that each $V_i$ induces a digraph of minimum out-degree at least~$i$. With such a partition in hand, one can then get the~$k$ desired disjoint cycles by just considering each of the $V_i$'s successively, and picking, in each of the digraphs they induce, one cycle whose length is different from the lengths of the previously picked cycles. This is possible according to Proposition~\ref{lemma:k-cycle-lengths} due to the out-degree property of the partition $V_1,...,V_k$.

We first introduce some notation and assumptions we use throughout (and further) to deal with our computations. Every parameter has to be thought of as a function of $k$. By writing $o(1)$, we refer to a term tending to $0$ as $k$ tends to infinity. Given two terms $a$ and $b$, we write $a \thicksim b$ if $a/b$ tends to $1$, and $a \ll b$ if $a/b$ tends to $0$ as $k$ tends to infinity. Let $k':=k+\lfloor k^{2/3}(\ln k)^{1/3}\rfloor$ and $s$ be the sum of the first $k'$ integers, that is $s:=\frac{k'(k'+1)}{2}$. It is assumed throughout that $r \thicksim s \thicksim \frac{k^2}{2}$; so $r$ and $s$ will sometimes freely be replaced by $\frac{k^2}{2}$ to simplify some computation (assuming this has no impact on the computation).

We now show that, under our assumptions, the desired partition $V_1,...,V_k$ of $V(D)$ exists with non-zero probability. For this purpose, let us just randomly $k$-color the vertices of $D$, where assigning color~$j$ to some vertex means that we put it into $V_j$. All colors are not assigned uniformly, but in such a way that, for every color $j \in \{1,...,k\}$, the probability $p_j$ that some vertex is colored~$j$ is:
\begin{displaymath}
p_j := \frac{j+\lfloor k^{2/3}(\ln k)^{1/3}\rfloor}{s}.
\end{displaymath}
Note that $\sum_{j=1}^k{p_j}<1$, therefore a vertex gets no color with probability $1-\sum_{j=1}^k{p_j}$.

Let $X_v^j$ denote the number of out-neighbors of $v$ colored $j$ by the random process. Clearly $X_v^j \thicksim BIN(r,p_j)$. So that we can later apply the Lov\'asz Local Lemma, let us define our set of bad events. Let $A_v$ be the event $$A_v := \bigcup_{j=1}^{k}(v \textrm{ is colored } j \textrm{ and } X_v^j<j ).$$ Second, let $B$ be the event that at least one of the~$k$ colors does not appear among a fixed subset $\{u_1, u_2, ..., u_r\}$ of $r$ vertices of $D$. It should be clear that any two events $A_v$ and $A_u$ are dependent if $(u \cup N^+(u)) \cap (v \cup N^+(v)) \ne \emptyset$. So, since $D$ is $r$-regular, each $A_v$ depends on at most $r^2$ other bad events $A_u$. Since the event $B$ only depends on the colors of~$r$ fixed vertices, similarly $B$ depends on at most $r^2$ other bad events. To apply the Lov\'asz Local Lemma, every bad event $A \in (\cup_{v}A_v)\cup B$ must hence fulfill $4r^2\Pr(A)\leq 1$.

Concerning $B$, we have:
\begin{align*}
\Pr(B) \leq \sum_{j=1}^{k}(1-p_j)^r &\leq k \left(1-\frac{\lfloor k^{2/3} (\ln k)^{1/3} \rfloor+1}{s} \right)^r \\
&\leq ke^{-\frac{k^{2/3} (\ln k)^{1/3} r}{s}} \\
&\leq ke^{- k^{2/3} (\ln k)^{1/3}}.
\end{align*}

\noindent Therefore, we have that:

\begin{equation*}
4r^2 \Pr(B) \leq 4\left(\frac{k^2}{2}+ \frac{7}{2}k^\frac{5}{3}(\ln k)^\frac{1}{3}  \right)^2ke^{-k^{2/3} (\ln k)^{1/3}} \leq 1.
\end{equation*}

 Now consider the $A_v$'s. Since
$$ s=\frac{(k')^2}{2}\left(1+\frac{1}{k'}\right) \leq \frac{k^2}{2}\left(1+\left(\frac{\ln k}{k}\right)^\frac{1}{3}\right)^2\left(1+\frac{2}{k}\right),  $$
\noindent then
$$ \left(1-4\left(\frac{\ln k}{k}\right)^\frac{1}{3}\right)rp_j \geq \left(1-4\left(\frac{\ln k}{k}\right)^\frac{1}{3}\right)\frac{rj}{s} \geq j.   $$

\noindent Applying Chernoff's Inequality, we get:
\begin{align*}
\Pr( X_v^j<j ) &\leq \Pr\left( X_v^j<\left(1-4\left(\frac{\ln k}{k}\right)^\frac{1}{3}\right)rp_j \right) \\
&\leq e^{-16\left(\frac{\ln k}{k}\right)^\frac{2}{3}\frac{rp_j}{3}}.
\end{align*}

\noindent Then:
\begin{displaymath}
\Pr(A_v) \leq \sum_{j=1}^{k}{p_j e^{-16\left(\frac{\ln k}{k}\right)^\frac{2}{3}\frac{rp_j}{3}}}.
\end{displaymath}

We hence want every term to be smaller than $\frac{1}{4r^2k}$, which is achieved as soon as it is verified for $p_j = \frac{\lfloor k^{2/3}(\ln k)^{1/3}\rfloor+1}{s}$. Indeed:
$$ \frac{\lfloor k^{2/3}(\ln k)^{1/3}\rfloor+1}{s} e^{-16\left(\frac{\ln k}{k}\right)^\frac{2}{3}\frac{r}{3}\frac{\lfloor k^{2/3}(\ln k)^{1/3}\rfloor+1}{s}} \leq \frac{1}{4r^2k}
\Leftrightarrow  e^{\frac{16}{3}\ln k} \geq 2k^\frac{11}{3}(\ln k)^\frac{1}{3}.  $$

Under all these conditions, the requirements for applying the Lov\'asz Local Lemma are met; we can hence deduce the desired partition $V_1,...,V_k$, thus the claimed $k$ disjoint cycles of distinct lengths.
\end{proof}

%%%%%%%%%%%%%%%%%%%%%%%%
%%%%%%%%%%%%%%%%%%%%%%%%
%%%%%%%%%%%%%%%%%%%%%%%%
%%%%%%%%%%%%%%%%%%%%%%%%
%%%%%%%%%%%%%%%%%%%%%%%%

\subsection{Small digraphs} \label{section:small}

We now prove Conjecture~\ref{conjecture:digraph-lengths} for digraphs whose order can be expressed as some particular function of the minimum out-degree.

\begin{theorem} \label{theorem:order}
Let $k \geq 1$. Then every simple digraph of order at most $c^{r^d}$ (where $c$ and $d$ are two constants satisfying $c>1$ and $0<d <1$) and minimum out-degree $r \geq c_0\max\{k^{\frac{1}{1-d}}, k^2\}$ contains at least $k$ vertex-disjoint directed cycles of distinct lengths, where $c_0=\max\{2,(24\ln c)^\frac{1}{1-d}\}$.
\end{theorem}

\begin{proof}
Let $D$ be such a digraph with order $n$ where we assume the out-degree of every vertex is exactly~$r$. We herein reuse the terminology introduced in the proof of Theorem~\ref{theorem:regular}.

The proof is similar to that of Theorem~\ref{theorem:regular}, except that the random $k$-coloring of the vertices of $D$ is this time performed uniformly (all $k$ colors have the same probability to be assigned). Given any vertex $v$ of $D$, let $X_v$ denote the number of out-neighbors of $v$ being assigned the same color by the random coloring. Clearly $X_v \thicksim BIN(r,1/k)$. Our two kinds of bad events are the following. First $A_v$ is, for every vertex $v$, the event that $X_v<k$. Second, let $B$ be the event that at least one of the $k$ colors does not appear at all.

We first assume the following:

\begin{equation} \label{E:2:4}
\frac{r}{2k} \geq k \Leftrightarrow r\geq 2k^2.
\end{equation}

\noindent Therefore, applying Chernoff's Inequality, we have:
$$ \Pr(A_v) = \Pr(X_v<k) \leq \Pr\left(X_v< \frac{r}{2k}\right) \leq e^{-\frac{r}{12k}}.$$

\noindent By the Union Bound, we deduce:
\begin{align} \label{eq1}
\Pr\left(\sum_v{A_v} + B\right) \leq \sum_v{\Pr(A_v)} + \Pr(B) & \leq ne^{-\frac{r}{12k}} + k \left(1-\frac{1}{k}\right)^n \nonumber \\
 & \leq ne^{-\frac{r}{12k}} + k\left(1-\frac{1}{k}\right)^r \nonumber \\
 & \leq ne^{-\frac{r}{12k}} + ke^{-\frac{r}{k}}.
\end{align}

\noindent Since we want this obtained sum to be smaller than $1$, we would like each of its two terms to be strictly smaller that $1/2$. Then, concerning the first term:

\begin{equation*}
ne^{-\frac{r}{12k}} < \frac{1}{2} \Leftrightarrow \frac{r}{12k} > \ln(2n).
\end{equation*}

\noindent Since $n\leq c^{r^d}$, we have $\ln(2n) \leq 2r^d\ln c$. So we need a stronger inequality:

\begin{equation} \label{E:2:5}
\frac{r}{12k} \geq 2r^d\ln c \Leftrightarrow r \geq (24k\ln c)^{\frac{1}{1-d}}.
\end{equation}

\noindent Concerning the second term of Inequality~(\ref{eq1}), we want:

\begin{equation} \label{E:2:6}
ke^{-\frac{r}{k}} < \frac{1}{2} \Leftrightarrow r > k\ln(2k).
\end{equation}

Suppose now that Inequalities~(\ref{E:2:5}) and (\ref{E:2:6}) hold. Then there exists a partition $V_1, ..., V_k$ of $V(D)$ such that the out-degree of each vertex in the part containing it is at least $k$. We are now able to successfully pick a cycle from each of these parts in such a way that all picked cycles have distinct lengths (by applying Proposition~\ref{lemma:k-cycle-lengths}).

So that all conditions of Inequalities~(\ref{E:2:4}), (\ref{E:2:5}) and (\ref{E:2:6}) are met, we then just need $r \geq c_0\max\{k^{\frac{1}{1-d}}, k^2\}$, where $c_0=\max\{2,(24\ln c)^\frac{1}{1-d}\}$, as claimed.
\end{proof}

\subsection{Concluding remarks} \label{section:ccl-digraphs}

In Sections~\ref{section:regular} and~\ref{section:small}, we have proved that Conjecture~\ref{conjecture:digraph-lengths} holds 
for regular digraphs (Theorem~\ref{theorem:regular}) and digraphs with bounded order (Theorem~\ref{theorem:order}). About these two results, let us mention the following:

\begin{enumerate}
    \item In the proofs of Theorems~\ref{theorem:regular} and~\ref{theorem:order}, if we require the parts of the partition $V_1,...,V_k$ to include more vertices, then we can deduce $k$ distinct cycles whose lengths are 'more than just distinct'. Possible such additional properties are \textit{e.g.} the cycles to be of even or odd lengths only, or to be of lengths divisible by some fixed integer, etc.

    \item In the statement of Theorem~\ref{theorem:order}, it is worth mentioning that we can lower the requirement on the out-degree if we assume a smaller upper bound on $n$. For example, if $n$ is bounded above by a polynomial function of $r$ ($n < r^d$ for some constant $d > 0$), then we can require the digraph to have minimum out-degree at least $k^2(\ln k)^3$ only (which sticks closer to what is stated in Conjecture~\ref{conjecture:digraph-lengths}).

    \item In case Conjecture~\ref{conjecture:digraph-lengths} is false and some counterexamples exist, Theorem~\ref{theorem:order} gives a lower bound on the order of these counterexamples. These should be of large order.
\end{enumerate}

\end{document}